\RequirePackage[british]{babel}
\documentclass[a4paper,leqno,twoside]{amsart}
\usepackage{amssymb}
\usepackage{amscd}
\usepackage{url}
\usepackage{tabularx}

\theoremstyle{plain}
\newtheorem{thm}{Theorem}
\newtheorem{cor}[thm]{Corollary}

\newtheorem{lem}[thm]{Lemma}
\newtheorem{conj}{Conjecture}

\newtheorem{fact}{Fact}

\theoremstyle{definition}

\theoremstyle{remark}
\newtheorem{rmk}{Remark}
\newtheorem*{acks}{Acknowledgements}

\newcommand{\ie}{\textit{i.e. }}

\newcommand{\p}{\mathbb{P}}

\newcommand{\OO}{\mathcal{O}}

\newcommand{\II}{\mathcal{I}}

\DeclareMathOperator{\length}{length}

\DeclareMathOperator{\sm}{sm}

\def\cocoa{{\hbox{\rm C\kern-.13em o\kern-.07em C\kern-.13em o\kern-.15em A}}}

\begin{document}

\title[On the quadratic normality and the triple curve]
{On the quadratic normality and the triple curve of three dimensional
subvarieties of $\p^5$}
\author{Pietro De Poi, Emilia Mezzetti \and Jos\'e Carlos
Sierra}

\address{\hskip -.43cm Pietro De Poi, Dipartimento di Sistemi e
istituzioni per l'economia, Universit\`a degli Studi dell'Aquila, Piazza del Santuario
19, I-67040 Roio Poggio (AQ), Italy. e-mail \url{depoi@dmi.units.it}}

\address{\hskip -.43cm Emilia Mezzetti, Dipartimento di Matematica e
Informatica, Universit\`a degli Studi di Trieste, Via Valerio
12/1, I-34127 Trieste, Italy. e-mail \url{ mezzette@units.it}}

\address{\hskip -.43cm Jos\'e Carlos Sierra, Departamento de
\'Algebra, Facultad de Ciencias Matem\'aticas, Universidad
Complutense de Madrid, 28040 Madrid, Spain. e-mail \url{
jcsierra@mat.ucm.es}}

\thanks{This research was partially supported by MiUR,
project  ``Geometria delle variet\`a algebriche e dei loro spazi di
moduli'' for the first two authors, by Regione Friuli Venezia
Giulia, ``Progetto D4'' for the first author, by Indam project
``Birational geometry of projective varieties'' for the second one,
and by the Spanish MEC project MTM2006-04785 for the third one}


\date{\today}

\begin{abstract}
A well-known conjecture asserts that smooth threefolds
$X\subset\p^5$ are quadratically normal with the only exception of
the Palatini scroll. As a corollary of a more general statement we
obtain the following result, which is related to the previous
conjecture: If $X\subset\p^5$ is not quadratically normal, then its
triple curve is reducible. Similar results are also given for higher
dimensional varieties.
\end{abstract}

\maketitle


\section{Introduction}

Let $X\subset\p^r$ be a reduced irreducible complex projective
subvariety. $X$ is said to be $k$-normal if the restriction map
$H^0(\p^r,\OO_{\p^r}(k))\to H^0(X,\OO_X(k))$ is surjective or,
equivalently, if $H^1(\p^r,\II_X(k))=0$.

A conjecture of Peskine and Van de Ven states that a smooth
threefold $X\subset\p^5$ is $2$-normal unless it is the Palatini
scroll (see, for instance, \cite[Problem 5]{schneider}). We remark
that smoothness cannot be dropped according to \cite{dp-m}. On the
other hand, the Palatini scroll is also the only known example of
smooth threefold $X$ in $\p^5$ with reducible triple curve (cf.
\cite{m-p}).

Let us define the triple locus. A line $L\subset\p^5$ is said to be
\emph{$k$-secant} to $X\subset\p^5$ (respectively, \emph{strict}
$k$-secant) if the scheme-theoretic intersection $X\cap L$ has
length at least $k$ (respectively, equal to $k$). We define the
\emph{triple locus} $\Gamma_P\subset X$ as the subvariety of points
contained in a $3$-secant line to $X$ passing through a point $P\in\p^5$. We say
that the triple locus of $X$ is irreducible if $\Gamma_P$ is
irreducible for a general $P\in\p^5$, and we say that an irreducible
triple locus is not quadruple if the general $3$-secant line to $X$
passing through $P$ is a strict $3$-secant line.

The main aim of this  note is to prove the following result:

\begin{thm}\label{thm:main}
Let $X\subset\p^5$ be an integral subvariety of dimension three. If the
triple locus of $X$ is non-empty, irreducible and not quadruple,
then $X$ is $2$-normal.
\end{thm}

Moreover, we say that $L$ is a \emph{true} $k$-secant line if
$X_{\sm}\cap L$ has length at least $k$, where $X_{\sm}\subset X$ is
the open subset of the smooth points. We also define the \emph{true
triple locus} $\gamma_P\subset X$ as the subvariety obtained by
taking the closure of the set of points contained in a true
$3$-secant line to $X$ passing through $P$, and we say, as before,
that the true triple locus of $X$ is irreducible if $\gamma_P$ is
irreducible for a general $P\in\p^5$.

We prove in Lemma \ref{lem:n-uple} that the true triple locus
$\gamma_P\subset X$ is at most $1$-dimensional for a general
$P\in\p^5$. Furthermore, if $\gamma_P\subset X$ is actually
$1$-dimensional we show in Lemma \ref{lem:strict} that all but a finite number of
true $3$-secant lines passing through $P$ are strict. Both lemmas
easily follow from the generalization of the classical Trisecant
Lemma given in \cite[Theorem 1]{ran}. Therefore, we finally define
the \emph{true triple curve of $X$} as the $1$-dimensional
component of $\gamma_P$. Then we can slightly modify the proof of
Theorem \ref{thm:main} to obtain the following:

\begin{cor}\label{cor:main}
Let $X\subset\p^5$ be an integral subvariety of dimension three. If the true
triple curve of $X$ is non-empty and irreducible, then $X$ is
$2$-normal.
\end{cor}

If moreover $X$ is smooth then its (true) triple curve is actually
non-empty unless $X\subset\p^5$ is contained in a quadric
hypersurface (see \cite{kwak}), so Corollary \ref{cor:main} yields:

\begin{cor}\label{cor:smooth}
Let $X\subset\p^5$ be a smooth threefold. If $X$ is not $2$-normal,
then the triple curve of $X$ is reducible.
\end{cor}

Hence in order to prove the conjecture of Peskine and Van de Ven, it
is enough to prove the following:

\begin{conj}
The only smooth threefold $X\subset\p^5$ with reducible triple
curve is the Palatini scroll.
\end{conj}

Finally, in Section \ref{fr}, Theorem \ref{thm:main} and its corollaries
are generalized to varieties of higher dimension.

\section{A geometric point of view}

To begin with, we present a geometric proof of the fact that if
$X$ in $\p^4$ is an integral surface with irreducible double curve, then $X$
is linearly normal. This follows as a consequence of the classical
characterizations of the Veronese surface given by Severi
\cite{severi} and Franchetta \cite{franchetta}, where, in the
latter, $X$ is supposed to have as singularities at most a finite
number of \emph{improper double points} (\ie the tangent cone is
the union of two planes in general position). This fact was proved
in \cite{tour} under Franchetta's assumption by using the monoidal
construction.

The double locus $\Delta_Q\subset X$ is defined as the set of points
contained in a $2$-secant line to $X$ passing through a general
point $Q\in\p^4$. This locus is $1$-dimensional whenever
$X\subset\p^4$ is non-degenerate, but it is not necessarily
equidimensional since smoothness of $X$ is not required.
Furthermore, if $\Delta_Q$ is irreducible the general $2$-secant
line to $X$ passing through $Q$ is a strict $2$-secant line by the
well-known Trisecant Lemma.

\begin{fact}\label{f:1}
Let $X\subset\p^4$ be an integral surface. If the double locus of
$X$ is irreducible, then $X$ is linearly normal.
\end{fact}

\begin{proof}
We can assume $X\subset\p^4$ non-degenerate. To get a contradiction,
let us suppose that $X\subset\p^4$ is not linearly normal. Then
there exists a linear projection
\[
\pi_P\colon \p^5\dashrightarrow\p^4
\]
from a point $P\in\p^5$ and a non-degenerate surface
$X'\subset\p^5$ such that the restriction map
\[
\pi_P\mid_{X'}\colon X'\to X
\]
is an isomorphism. In particular, $P$ does not belong to any
$2$-secant line to $X'$. Let $SX'\subset\p^5$ denote the $2$-secant
variety defined as the locus of points in $\p^5$ contained in a
$2$-secant line to $X'$. We have $\dim(SX')\leq 4$ since
$\pi_P\mid_{X'}\colon X'\to X$ is an isomorphism. On the other hand
$\dim(SX')\geq 4$, since a $3$-dimensional variety containing a
$4$-dimensional family of lines is necessarily $\p^3$ and, in that
case, $X'\subset\p^5$ would be degenerate. Hence $\dim(SX')=4$. Let
$W\subset SX'$ denote the irreducible $4$-dimensional component
corresponding to the closure of the union of lines joining two
different points of $X'$, and let $d>1$ denote its degree in $\p^5$.
Note that through any point of $W$ there pass infinitely many
$2$-secant lines to $X'$.

Let $Q\in\p^4$ be a general point and let $\pi_P^{-1}(Q)\cap W
=\{Q_1,\dots,Q_d\}\subset\p^5$. We have $Q_i\neq Q_j$ for $i\neq j$
and $Q_i\notin X'$, since $Q$ is general. For $i\in\{1,\dots,d\}$ we
denote by $\Delta'_{Q_i}\subset X'$ the double locus corresponding
to the points contained in a $2$-secant line to $X'$ passing through
$Q_i$. Since $Q_i\notin X'$ the family of $2$-secant lines to $X'$
passing though $Q_i$ cannot be $2$-dimensional, so $\Delta'_{Q_i}$
is actually $1$-dimensional.

We claim that $\Delta'_{Q_{i_0}}\neq\Delta'_{Q_j}$ for some
${i_0}\in\{1,\dots, d\}$ and every $j\neq i_0$. First, we remark
that if $L\subset\p^4$ is a strict $2$-secant line to $X$, then
there exists a unique $2$-secant line $L'\subset\p^5$ to $X'$ such
that $\pi_P(L')=L$. This is obvious if $L$ intersects $X$ in two
different points. If $L\cap X$ is a double point then
$\pi_P\mid_{X'}^{-1}(L\cap X)$ is a double point, whence there
exists a unique line $L'\subset\p^5$ such that $L'\cap
X'=\pi_P\mid_{X'}^{-1}(L\cap X)$. Let us prove the claim. Since
$Q\in\p^4$ is a general point, there exists a general point
$Q_{i_0}\in W$ such that $\pi_P(Q_{i_0})=Q$. Assume by contradiction
that $\Delta'_{Q_{i_0}}=\Delta'_{Q_j}$, and let
$x'\in\Delta'_{Q_{i_0}}$ be a general point (of course
$\Delta'_{Q_{i_0}}$ may be reducible, but in this case so is
$\Delta_Q$). Then the lines $L'_{i_0}:=\langle x', Q_{i_0}\rangle$
and $L'_j:=\langle x', Q_j\rangle$ are $2$-secant lines to
$X'\subset\p^5$ such that $\pi_P(L'_{i_0})=\pi_P(L'_j)=\langle
\pi_P(x'),Q\rangle=:L$ is a $2$-secant line to $X\subset\p^4$.
Moreover $x'\in X'$ is a general point since $Q_{i_0}\in W$ is a
general point. Then $\pi_P(x')\in X$ is also a general point, and
hence $L\subset\p^4$ is a strict $2$-secant line to $X$. This
contradicts the previous remark, since
$\pi_P(L'_{i_0})=\pi_P(L'_j)=L$.

Since $\pi_P\mid_{X'}\colon X'\to X$ is an isomorphism we have
$\pi_P(\Delta'_{Q_{i_0}})\neq\pi_P(\Delta'_{Q_j})$ for every $j\neq
i_0$. Therefore $\Delta_Q\subset X$ has at least two irreducible
components, contradicting the hypothesis.
\end{proof}


\section{Proof of Theorem \ref{thm:main}}
We obtain our main result inspired by the geometric proof of Fact
\ref{f:1}.

\begin{proof}[Proof of Theorem \ref{thm:main}] We divide the
proof in three parts:

{\textbf Step 1:} \emph{Set up.} Let us consider the Veronese
embedding
\[
v_2:\p^5\to v_2(\p^5)\subset\p^{20}
\]
given by the complete linear system of quadrics in $\p^5$ and let
\[
v_2|_X: X\to v_2(X)\subset\p^{20}
\]
denote its restriction to $X$. Since the triple locus of
$X\subset\p^5$ is non-empty we get $H^0(\p^5,\II_X(2))=0$, whence
$v_2(X)\subset\p^{20}$ is non-degenerate.

To get a contradiction we assume that $X\subset\p^5$ is not
$2$-normal. This is equivalent to suppose that
$v_2(X)\subset\p^{20}$ is not linearly normal. Then there exists a
point $P\in\p^{21}$ and a non-degenerate threefold
$X'\subset\p^{21}$ such that the linear projection
\[
\pi_P:\p^{21}\dashrightarrow\p^{20}
\]
from $P$ induces by restriction an isomorphism
\[
\pi_P|_{X'}: X'\to v_2(X).
\]

{\textbf Step 2:} \emph{The variety of $3$-secant conics.} We denote
by $C_P(v_2(\p^5))\subset\p^{21}$ the $6$-dimensional cone of vertex
$P$ over $v_2(\p^5)\subset\p^{20}$. Let $\mathcal H$ be the Hilbert
scheme of conics in $\p^{21}$. We say that a conic $C\in\mathcal H$
is $3$-secant to $X'\subset\p^{21}$ if the scheme-theoretic
intersection $C\cap X'$ has length at least three. Furthermore, $C$
is said to be a strict $3$-secant conic if $\length(C\cap X')=3$.
Let $V\subset C_P(v_2(\p^5))$ denote the subvariety of points lying
on a $3$-secant conic to $X'$ which is contained in
$C_P(v_2(\p^5))$. We remark that $V$ plays the role of the
$2$-secant variety $SX'\subset\p^5$ in the proof of Fact \ref{f:1}.

We claim that $P\notin V$. Otherwise, there exists a (maybe
reducible or non-reduced) $3$-secant conic $C$ to $X'$ passing
through $P$ such that $C\subset C_P(v_2(\p^5))$. Let
$\p^2_C\subset\p^{21}$ denote the linear span of $C$. Then,
$\pi_P|_{X'}: X'\to v_2(X)$ being an isomorphism, we get that
$\pi_P(\p^2_C)\subset\p^{20}$ is a $3$-secant line to $v_2(X)$ .
This is impossible since $v_2(\p^5)$, and hence $v_2(X)$, has no
$3$-secant lines.

It follows that $\dim V<6$. Let us see that $\dim V=5$ proving that
$\pi_P(V)$ contains $v_2(\p^5)$. Since the triple locus of
$X\subset\p^5$ is non-empty, there exists an irreducible family
$\{L_z\}_{z\in Z}$ of $3$-secant lines to $X$ that fill up $\p^5$.
We choose $Z$ of maximal dimension satisfying this property, so a
general $L_z$ is a strict $3$-secant line to $X$ by hypothesis. It
is enough to prove that a strict $3$-secant conic $v_2(L_z)$ to
$v_2(X)$ can be uniquely lifted to a $3$-secant conic $C_z$ to $X'$
such that $\pi_P(C_z)=v_2(L_z)$. Let $\xi$ be the scheme-theoretic
intersection $L_z\cap X$ of length three. Then $v_2(L_z)$ is a
strict $3$-secant conic to $v_2(X)$. We now consider the subscheme
$\xi':=\pi_P\mid_{X'}^{-1}(v_2(\xi))$ of $X'$ of length three. We
remark that there exists a unique plane $\p^2_{\xi'}$ such that
$\p^2_{\xi'}\cap X'=\xi'$. Then we define the conic
$C_z\subset\p^2_{\xi'}$ as the intersection of the quadric cone
$C_P(v_2(L_z))$ and $\p^2_{\xi'}$. Moreover, $\pi_P(C_z)=v_2(L_z)$
since $P\notin\p^2_{\xi'}$ by the same reason of the above claim.
This proves $v_2(\p^5)\subset\pi_P(V)$, whence $\pi_P(V)=v_2(\p^5)$.

{\textbf Step 3:} \emph{End of the proof.} Let $V'\subset V$ denote
the irreducible $5$-dimensional component arising from the union of
the conics corresponding to the family $Z$. The general point $x\in
X$ is contained in a $3$-secant line to $X$ since the family
$\{L_z\}$ fills up $\p^5$. Hence $v_2(x)\in v_2(X)$ is contained in a
$3$-secant conic to $v_2(X)$, so $x':=\pi^{-1}_P|_{X'}(v_2(x))\in
X'$ is a general point contained in a $3$-secant conic to $X'$.
Therefore $X'\subset V'$. Note that $V'\subset\p^{21}$ is
non-degenerate since $X'\subset\p^{21}$ is non-degenerate and
$X'\subset V'$. Hence the induced linear projection
\[
\pi_P\mid_{V'}\colon V'\to v_2(\p^5)
\]
cannot be an isomorphism since $v_2(\p^5)\subset\p^{20}$ is linearly
normal. We claim that moreover $\pi_P\mid_{V'}\colon V'\to
v_2(\p^5)$ cannot be a birational morphism. Otherwise, by Zariski's
Main Theorem, all its fibres should be connected. In particular, if
we consider two points with the same image, the corresponding fibre
should be the whole line through $P$, which is impossible since
$P\notin V'$.

We deduce that $\pi_P\mid_{V'}\colon V'\to v_2(\p^5)$ is a morphism
of degree $d>1$ by the previous claim.

Let $Q\in\p^5$ be a general point and let $\Gamma_Q\subset X$ denote
the corresponding triple locus. Consider $v_2(Q)\in v_2(\p^5)$ and
let
\[
\pi^{-1}_P(v_2(Q))\cap V'=\{Q_1,\dots,Q_d\}\subset\p^{21}.
\]
We remark that $Q_i\neq Q_j$ for $i\neq j$ and $Q_i\notin X'$, since
$Q\in\p^5$ is a general point. For $i\in\{1,\dots,d\}$ we denote by
$\Gamma_{Q_i}'\subset X'$ the locus of points lying on a $3$-secant
conic to $X'$ passing through $Q_i$ and contained in $V$.

We now check that $\Gamma_{Q_{i_0}}'\neq\Gamma_{Q_j}'$ for some
${i_0}\in\{1,\dots, d\}$ and every $j\neq i_0$. Since $Q\in\p^5$ is
a general point, there exists a general point $Q_{i_0}\in V'$ such
that $\pi_P(Q_{i_0})=v_2(Q)$. Assume to the contrary that
$\Gamma_{Q_{i_0}}'=\Gamma_{Q_j}'$ and let $x'\in\Gamma_{Q_{i_0}}'$
be a general point (as in Fact \ref{f:1}, $\Gamma'_{Q_{i_0}}$ may be
reducible but in this case so is $\Gamma_Q$). Then there exist two
$3$-secant conics $C_{i_0}$ and $C_j$, passing through $x'$ and
contained in $V'$, such that $Q_{i_0}\in C_{i_0}$ and $Q_j\in C_j$.
Let us see that $\pi_P(C_{i_0})=\pi_P(C_j)$. Consider the line
$L:=\langle x, Q\rangle\subset\p^5$, where $x=v_2^{-1}(\pi_P(x'))\in
X$. Then $v_2(L)$ is the only $3$-secant conic to $v_2(X)$ joining
$v_2(Q)$ and $\pi_P(x')$ and contained in $v_2(\p^5)$, so
necessarily $\pi_P(C_i)=\pi_P(C_j)=v_2(L)$. Moreover $L$ is a strict
$3$-secant line to $X\subset\p^5$ since $x'\in X'$, and hence
$\pi_P(x')\in v_2(X)$, is a general point. Therefore $v_2(L)$ is a
strict $3$-secant conic to $v_2(X)$, so it can be uniquely lifted to
a $3$-secant conic to $X'$ contained in $V$, as we showed in Step 2.
This yields a contradiction since $\pi_P(C_i)=\pi_P(C_j)=v_2(L)$.

It follows from $\Gamma_{Q_{i_0}}'\neq\Gamma_{Q_j}'$ that also
$\pi_P(\Gamma_{Q_{i_0}}')\neq\pi_P(\Gamma_{Q_j}')$, since
$\pi_P|_{X'}: X'\to v_2(X)$ is an isomorphism. So we finally deduce
that $\Gamma_Q\subset X$ has at least two irreducible components,
contradicting the hypothesis.
\end{proof}

\begin{rmk}
The converse of Theorem \ref{thm:main} does not hold since there
exist $2$-normal cones $X\subset\p^5$ with vertex a point over a
surface whose triple locus consists of a union of rulings of $X$.
\end{rmk}

We pass to prove Corollary \ref{cor:main}. We first show that the
true triple locus of $X$ is at most $1$-dimensional:

\begin{lem}\label{lem:n-uple}
Let $X\subset\p^{n+2}$ be a subvariety of dimension $n\geq 1$.
Then the family of true $n$-secant lines to $X$ passing through a
general point $P\in\p^{n+2}$ is at most $1$-dimensional.
\end{lem}

\begin{proof}
Assume that the family of true $n$-secant lines to $X$ passing
through a general $P\in\p^{n+2}$ is at least $2$-dimensional. Let
$Y=X\cap H_1\cap H_2$ be the intersection of $X$ with two general
hyperplanes of $\p^{n+2}$. Then the union of the true $n$-secant
lines to $Y$ fill up $\p^n$, contradicting \cite[Theorem 1]{ran}.
\end{proof}

Moreover, if the true triple locus is $1$-dimensional then there
exists a $5$-dimensional family $Z$ of true $3$-secant lines to
$X$ filling up $\p^5$. In this case the general element of $Z$ is
a strict $3$-secant line, as the following lemma shows:

\begin{lem}\label{lem:strict}
Let $X\subset\p^{n+2}$ be a subvariety of dimension $n\geq 1$. Let
$Z$ be an irreducible $(n+2)$-dimensional variety parametrising a
family $\{L_z\}$ of true $n$-secant lines to $X\subset\p^{n+2}$
whose points fill up $\p^{n+2}$. Then a general $L_z$ is a strict
$n$-secant line to $X$.
\end{lem}

\begin{proof}
Assume to the contrary that a general $L_z$ is an $(n+1)$-secant
line to $X$. Let $Y=X\cap H\subset\p^{n+1}$ be a hyperplane section
of $X\subset\p^{n+2}$. Then a general $P\in\p^{n+1}$ is contained in
a true $(n+1)$-secant line to $Y$ since $P$ is contained in
infinitely many $(n+1)$-secant lines to $X$, contradicting again
\cite [Theorem 1]{ran}.
\end{proof}

\begin{proof}[Proof of Corollary \ref{cor:main}]
If the true triple curve of $X$ is non-empty and irreducible, we can
repeat the proof of Theorem~\ref{thm:main} considering in Step 2 a
$5$-dimensional family $\{L_z\}$ of true $3$-secant lines filling up
$\p^5$, whose general element is a strict $3$-secant line by
Lemma~\ref{lem:strict}.
\end{proof}

We now show that Corollary \ref{cor:smooth} easily follows from
Corollary \ref{cor:main}:

\begin{proof}[Proof of Corollary \ref{cor:smooth}]
If the triple curve of $X$ is non-empty we apply Corollary
\ref{cor:main}. On the other side, if the triple curve of $X$ is
empty then $H^0(\p^5,\II_X(2))\neq 0$ by \cite[Lemma 3.3 and
Theorem 3.4~(a)]{kwak} since $X$ is smooth. Therefore
$X\subset\p^5$ is projectively normal as a consequence of
\cite[Theorem 1]{KW}.
\end{proof}

\section{Final remarks}\label{fr}

We would like to stress that the proof of Theorem \ref{thm:main}
can be extended to obtain a similar result for codimension two
subvarieties $X\subset\p^{n+2}$ of higher dimension:

\begin{thm}\label{thm:general}
Let $X\subset\p^{n+2}$ be an integral subvariety of dimension $n\geq 4$. If
the $n$-tuple locus of $X$ is non-empty, irreducible and not
$(n+1)$-tuple, then $X$ is $(n-1)$-normal.
\end{thm}

\begin{proof}
(Sketch) If $X\subset\p^{n+2}$ is not $(n-1)$-normal and the
$n$-tuple locus of $X$ is non-empty, then there exists a point $P$
and a non-degenerate embedding $X'\subset\p^{\binom{2n+1}{n-1}}$
such that the linear projection from $P$ to
$v_{n-1}(X)\subset\p^{\binom{2n+1}{n-1}-1}$ is an isomorphism. Then
we can repeat the proof of Theorem \ref{thm:main}, having in mind
that a strict $n$-secant rational normal curve $v_{n-1}(L)$ of
degree $n-1$ to $v_{n-1}(X)$, can be uniquely lifted to an
$n$-secant rational normal curve of degree $n-1$ to $X'$ contained
in the cone $C_P(v_{n-1}(\p^{n+2}))$.
\end{proof}

In view of Lemma \ref{lem:n-uple}, we can also define the true
$n$-tuple curve of $X\subset\p^{n+2}$ as the $1$-dimensional
component of the true $n$-tuple locus. Then the last assumption of
Theorem \ref{thm:general} holds by Lemma \ref{lem:strict}, and we
can also generalize Corollary \ref{cor:main} in the following way:

\begin{cor}\label{cor:general}
Let $X\subset\p^{n+2}$ be an integral subvariety of dimension $n\geq 4$. If
the true $n$-tuple curve of $X$ is non-empty and irreducible, then
$X$ is $(n-1)$-normal.
\end{cor}

According to Hartshorne's Conjecture, smooth subvarieties
$X\subset\p^{n+2}$ of dimension $n\geq 4$ are expected to be
complete intersections (and hence projectively normal). However, we
remark that Corollary \ref{cor:smooth} can be also extended to
higher dimensions thanks to \cite{ran2}:

\begin{cor}
Let $X\subset\p^{n+2}$ be a smooth $n$-fold, $n\geq 4$. If $X$ is
not $(n-1)$-normal, then the $n$-tuple curve of $X$ is reducible.
\end{cor}

\begin{proof}
If the $n$-tuple curve of $X$ is non-empty the result follows from
Corollary \ref{cor:general}. On the other hand, if the $n$-tuple
curve of $X$ is empty then it follows from \cite{ran2} that
$X\subset\p^{n+2}$ is a complete intersection. For the reader's
convenience, we explain this point in detail. Let $Y=X\cap H$ be a
general hyperplane section. It follows from Barth and Larsen's
theorems that $X\subset\p^{n+2}$ is subcanonical, whence
$Y\subset\p^{n+1}$ is also subcanonical. Note that the $n$-tuple
locus of $Y$ is empty since the $n$-tuple locus of $X$ is at most
$0$-dimensional. Then we apply \cite[Proposition]{ran2} to
$Y\subset\p^{n+1}$. Since $\Sigma_n=\emptyset$, it follows that
$e(k):=d-k\nu+k^2=0$ for some $k\in\{1,\dots,n-1\}$, where $d$ is
the degree of $Y\subset\p^{n+1}$ and $\bigwedge
N_{Y/\p^{n+1}}\cong\OO_Y(\nu)$. Then $\nu=\frac{d}{k}+k$. We
remark that \cite[Theorem]{ran2} holds under the weaker hypothesis
$\nu\geq\frac{d}{\alpha}+\alpha$ for some $\alpha\in (0,n-1]$, as
the author pointed out in the beginning of the proof. Therefore
$Y\subset\p^{n+1}$ (and hence $X\subset\p^{n+2}$) is a complete
intersection.
\end{proof}

\begin{acks} The authors wish to thank Fyodor Zak for let them know
the unpublished article \cite{zak}. The first two authors would like
to thank Fabrizio Catanese for interesting discussions. The third
author wishes to thank the warm hospitality of the Department of
Mathematics when he visited the University of Trieste on February
2007, as well as Angelo F. Lopez for introducing him to this topic.
Finally, the authors would like to thank the anonymous referee for
many helpful comments and suggestions. \end{acks}


\end{document}